\documentclass[a4paper,12pt]{article}

\usepackage{amsmath ,amsfonts,  amssymb, amscd, amsthm} 
\usepackage{bbm}
\usepackage{mathrsfs}
\usepackage[font=small,skip=0pt]{caption}

\usepackage{csquotes, color}
\usepackage{colortbl}
\usepackage[numbers]{natbib}
\usepackage{verbatim}
\usepackage[toc,page]{appendix}
\usepackage{caption}
\usepackage{textcomp}
\usepackage[inline]{enumitem}
\usepackage{tcolorbox}
\usepackage{cancel}
\usepackage{relsize}
\usepackage{bm}

\usepackage{authblk}

\usepackage{algorithm}
\usepackage{algorithmic}
\usepackage{scrextend}
\addtokomafont{labelinglabel}{\sffamily}

\usepackage{tikz}
\usepackage{tikz-cd}
\usetikzlibrary{arrows.meta}
\usetikzlibrary{positioning}
\usetikzlibrary{shadows}
\usetikzlibrary{calc, intersections}
\usetikzlibrary{shapes.multipart}
\usetikzlibrary{patterns}

\newtheorem{theorem}{Theorem}
\newtheorem{corollary}{Corollary}[theorem]

\newtheorem{proposition}{Proposition}

\newtheorem{definition}{Definition}
\newtheorem*{definition*}{Definition}

\newtheorem*{example*}{Example}

\newtheorem*{conditional*}{Conditional}

\newtheorem*{remark*}{Remark}

\tikzset{Myarrow/.style={very thin, arrows=-Latex}}

\definecolor{Gray}{gray}{0.85}
\definecolor{LightCyan}{rgb}{0.88,1,1}

\newcolumntype{a}{>{\columncolor{Gray}}c}
\newcolumntype{b}{>{\columncolor{Gray}}r}
\newcolumntype{d}{>{\columncolor{white}}c}

\newcommand{\scrL}{{\mathscr L}}

\usepackage{hyperref}
\hypersetup{
    colorlinks=true,
    citecolor=black,
    linkcolor=black,
    filecolor=cyan,      
    urlcolor=blue,
}
 
\providecommand{\keywords}[1]
{
  \small	
  \textbf{\textit{Keywords---}} #1
}

\title{
On Feller, Pollard and the Complete Monotonicity 
of the Mittag-Leffler Function $E_\alpha(-x)$
}
\author{Nomvelo Karabo Sibisi \\{\small {\tt sbsnom005@myuct.ac.za}}}
\date{\today}

\begin{document}
\maketitle
\thispagestyle{empty}


\begin{abstract}
\noindent
Pollard used contour integration 
to  show that  the Mittag-Leffler function is the Laplace transform of a positive function,
thereby proving that 
it is completely monotone.
He  also cited personal communication by Feller of a 
discovery of the result by  ``methods of probability theory''.
In his published work, Feller 
used the two-dimensional Laplace transform of a  bivariate distribution to derive the Pollard result.
But both approaches may be described as  analytic,
despite the occurrence  of the  stable distribution in Feller's starting point and in the Pollard result itself.
We adopt  a Bayesian probabilistic approach that assigns  a prior distribution to the scale parameter of the stable distribution.
We present  Feller's method as a particular  instance of such assignment.
The  Bayesian framework  enables generalisation of 
the Pollard result.
This leads to a novel integral representation of the Mittag-Leffler function as well as a
variant   arising from polynomial tilting of the stable density.
\end{abstract}

\section{Background}
\label{sec:intro}

An infinitely differentiable function $\varphi(x)$ on $x>0$ is completely monotone  if its derivatives $\varphi^{(n)}(x)$ 
satisfy $(-1)^n\varphi^{(n)}(x)\ge0$, $n\ge 0$.
Bernstein's theorem 
states that $\varphi(x)$ is completely monotone iff it may be expressed as 
\begin{align}
\varphi(x) &= \int_0^\infty e^{-x t}\,dF(t) 
= \int_0^\infty e^{-x t} f(t)dt
\label{eq:LT} 
\end{align}
for  a non-decreasing distribution function $F(t)$ with density $f(t)$, {\it i.e.}\ $F(t)=\int_0^t f(u)du$.
The first integral in~(\ref{eq:LT}) is formally called the Laplace-Stieltjes transform of $F$ and  
the latter the (ordinary) Laplace transform  of  $f$.
For bounded $F(t)$, $\varphi(x)$ is defined on $x\ge0$.
Integrating~(\ref{eq:LT}) by parts in this case gives  $\varphi(x)$ in terms of 
the ordinary Laplace transform of $F$:
\begin{align}
   \varphi(x) &=  x \int_{0}^\infty e^{-xt} F(t)\,dt
\end{align}
The Mittag-Leffler function 
$E_\alpha(x)$  is defined by  the infinite series 
\begin{align}
E_\alpha(x) &= \sum_{k=0}^\infty \frac{x^k}{\Gamma(\alpha k+1)} \quad \alpha\ge0
\label{eq:ML}
\end{align}
For later reference, the Laplace transform of 
$E_\alpha(-\lambda x^\alpha)$ $(\lambda>0)$ is 
\begin{align}
\int_0^\infty e^{-sx} E_\alpha(-\lambda x^\alpha) \,dx  &= \frac{s^{\alpha-1}}{\lambda+s^\alpha} \qquad {\rm Re}(s)\ge0
\label{eq:LaplaceML}
\end{align}
Pollard and Feller discussed the complete monotonicity of $E_\alpha(-x)$ from different perspectives.
We summarise both  
before presenting a Bayesian argument.

\subsection{Pollard's Approach}
\label{sec:Pollard}

In a 1948 paper, Pollard~\cite{PollardML}  led with the  opening remark:
\begin{quote}
``W.~Feller communicated to me his discovery -- by the methods of probability theory -- that if $0\le \alpha \le1$ 
the function $E_\alpha(-x)$ is completely monotonic for $x\ge0$. 
This means that it can be written in the form
\begin{align*}
E_\alpha(-x) &= \int_{0}^\infty e^{-xt} dP_\alpha(t) 
\end{align*}
where $P_\alpha(t)$ is nondecreasing and bounded.
In this note we shall prove this fact directly  and determine the function $P_\alpha(t)$ explicitly.'' \newline
 [we use $P_\alpha$ where Pollard used $F_\alpha$, which we reserve for another purpose]
\end{quote}
Having dispensed with  $E_0(-x)=1/(1+x)$ and $E_1(-x)=e^{-x}$ since ``there is nothing to be proved in these cases'',
Pollard used  a contour integral representation of $E_\alpha(-x)$:
\begin{align}
 E_\alpha(-x) &= \frac{1}{2\pi i}\oint_{C} \frac{s^{\alpha-1}e^s}{x+s^{\alpha}}\,ds  
                       = \frac{1}{2\pi i\alpha}\oint_{C^\prime} \frac{e^{z^{\frac{1}{\alpha}}}}{x+z}\,dz
\label{eq:MLcontour}
\end{align}
to prove that 
\begin{alignat}{3}
p_\alpha(t) &\equiv P_\alpha^{\,\prime}(t) = \frac{1}{\alpha}\, f_\alpha(t^{-1/\alpha})\, t^{-1-1/\alpha}  \qquad && 0<\alpha<1
\label{eq:Pollarddensity}
\intertext{where $f_\alpha(t)$ is defined by} 
e^{-s^\alpha}  &= \int_0^\infty e^{-s t} f_\alpha(t)\,dt  && 0<\alpha<1 
\label{eq:stable}
\end{alignat}
Pollard~\cite{Pollard} had earlier proved that $f_\alpha(t)>0$, so that
$p_\alpha(t)\ge0$, thereby completing his proof that  $E_\alpha(-x)$ is completely monotone for $0\le \alpha \le1$.
Pollard stopped at the point of deriving~(\ref{eq:Pollarddensity}), the density  $p_\alpha(t)\equiv P_\alpha^{\,\prime}(t)$.
As per initial task, we proceed to discuss $P_\alpha(t)$ explicitly.
We first recognise $f_\alpha(t)$ as  the density  of the stable distribution $F_\alpha$ on $[0,\infty)$
\begin{alignat}{3}
F_\alpha(t) &= \int_0^t f_\alpha(u)\,du  \qquad && 0<\alpha<1 
\label{eq:stableF}
\end{alignat}
with normalisation $F_\alpha(\infty)=1$.
In turn, $P_\alpha$ is the distribution
\begin{align}
P_\alpha(t) &= \int_0^t p_\alpha(u)\,du  = \frac{1}{\alpha} \int_0^t  f_\alpha(u^{-1/\alpha})\, u^{-1-1/\alpha} \, du \nonumber 
\intertext{Setting $y=u^{-1/\alpha}$ gives a simple relation between $P_\alpha$ and $F_\alpha$:}
P_\alpha(t)  &= \int_{t^{-1/\alpha}}^\infty  f_\alpha(y)\, dy 
 =  1- \int_0^{t^{-1/\alpha}} f_\alpha(y)\, dy   
  \equiv 1-  F_\alpha(t^{-1/\alpha}) 
\label{eq:PollardP}
\end{align}
Setting aside Pollard's contour integral proof, it is not clear how  to show directly that 
\begin{align}
E_\alpha(-x) 
= \int_{0}^\infty e^{-xt} dP_\alpha(t)  
&= \int_{0}^\infty e^{-xt} \, d(1-  F_\alpha(t^{-1/\alpha})) \nonumber \\
   &= x \int_{0}^\infty e^{-xt} (1-  F_\alpha(t^{-1/\alpha}) )\,dt
\label{eq:LaplacePollardP}
\end{align}
Feller followed a different route, discussed next.  

\subsection{Feller's Approach}
\label{sec:Feller}

In an illustration of the use of the two-dimensional Laplace transform, 
Feller~\cite{Feller2}(p453) considered  $1-F_\alpha(xt^{-1/\alpha})$,
a bivariate generalisation of~(\ref{eq:PollardP}) over $x>0,t>0$.
The Laplace transform over $x$, followed by that over  $t$ gives
\begin{align}
\int_{0}^\infty e^{-sx} (1-F_\alpha(xt^{-1/\alpha}))\,dx &= \frac{1}{s}-\frac{e^{-ts^\alpha}}{s}
\label{eq:FellerLaplace1} \\
\frac{1}{s}  \int_{0}^\infty e^{-\lambda t} \left(1-e^{-ts^\alpha}\right)dt &= \frac{1}{\lambda} \frac{s^{\alpha-1}}{\lambda+s^\alpha}
 \label{eq:FellerLaplace2}
\end{align}
By reference to~(\ref{eq:LaplaceML}), the right hand side of~(\ref{eq:FellerLaplace2}) is the 
 Laplace transform of $E_\alpha(-\lambda x^\alpha)/\lambda$.
Since the two-dimensional Laplace transform equivalently can  be evaluated first over $t$ then over $x$,
Feller concluded that 
\begin{align}
E_\alpha(-\lambda x^\alpha) &= \lambda \int_{0}^\infty e^{-\lambda t} (1-F_\alpha(xt^{-1/\alpha}))\,dt 
\label{eq:FellerPollard1} \\
\implies E_\alpha(-x) &= \int_{0}^\infty e^{-t} (1-F_\alpha(x^{1/\alpha}t^{-1/\alpha}))\,dt \\
(t\to xt)\qquad     &= x \int_{0}^\infty e^{-xt} (1-F_\alpha(t^{-1/\alpha}))\,dt
\label{eq:FellerPollard2}
\end{align}
which is the Pollard result in the form~(\ref{eq:LaplacePollardP}).

Feller's proof  is based on the interchange of the order of integration (Fubini's theorem) and the 
uniqueness of Laplace transforms.
It can  be represented by the commutative diagram below, where $\scrL_{s\vert t}$ denotes
the one-dimensional Laplace transform of a bivariate source function at fixed $t$, 
to give 
a  bivariate function of $(s,t)$ where $s$ is the Laplace variable.
\begin{equation}
\begin{tikzpicture}[auto,scale=1.8, baseline=(current  bounding  box.center)]
\newcommand*{\size}{\scriptsize}%
\newcommand*{\gap}{.2ex}%
\newcommand*{\width}{3.0}%
\newcommand*{\height}{1.5}%

\node (P) at (0,0)  {$1-F_\alpha(xt^{-1/\alpha})$};
\node (Q) at ($(P)+(\width,0)$) {$\dfrac{1}{s}-\dfrac{e^{-ts^\alpha}}{s}$};
\node (B) at ($(P)-(0,\height)$) {$\dfrac{1}{\lambda}E_\alpha(-\lambda x^\alpha)$}; 
\node (C) at ($(B)+(\width,0)$) {$\dfrac{1}{\lambda}\dfrac{s^{\alpha-1}}{\lambda+s^\alpha}$};   
\draw[Myarrow] ([yshift =  \gap]P.east)  --  node[above] {\size $\scrL_{s\vert t}$} node[below]{\size easy}  ([yshift = \gap]Q.west) ;
\draw[Myarrow]([xshift  =  \gap]Q.south) --  node[left] {\size $\scrL_{\lambda\vert s}$} node[right] {\size easy} ([xshift =  \gap]C.north);
\draw[Myarrow] ([xshift =  \gap]P.south) -- node[left] {\size $\scrL_{\lambda\vert x}$}  node[right] {\size hard}
([xshift =  \gap]B.north);
\draw[Myarrow] ([yshift = +\gap]C.west) --  node[above] {\size $\scrL^{-1}_{x\vert \lambda}$} node[below]{\size easy}  ([yshift  = +\gap]B.east); 
\end{tikzpicture}
\label{eq:Fellerdiagram}
\end{equation}
The desired  proof  is the   ``hard'' direct path, which is equivalent to the  ``easy'' indirect path.
We will return to commutative diagram representation in a different context later in the paper.

Feller's concise proof uses  ``methods of probability theory'', as cited  by Pollard, only to the extent 
of choosing the  bivariate distribution as input to the   two-dimensional  Laplace transform.
Other than that, the methods by both  Pollard and Feller might  be described as analytic rather than probabilistic.
This naturally  begs the following  questions:
\begin{enumerate}
\item What is it  that amounts to a  method of probability theory, at least in the context of proving that 
$E_\alpha(-x)$ is completely monotone?
\item What additional or complementary  insight, if any, does probabilistic reasoning offer relative to an analytic perspective?
\end{enumerate}

\subsection{Purpose of Paper}
\label{sec:purpose}
This paper addresses both questions above.
The  approach is  that of strict use of  the  sum and product rules of probability theory.
We identify this as  Bayesian reasoning, although our context is not one of  Bayesian inference.
The latter  calls for explicit use of Bayes' rule to transition from prior to posterior distribution,
with the aid of a prescribed likelihood.
The assignment of appropriate  distribution in our context is guided by the  
 task of proving  that $E_\alpha(-x)$ is completely monotone.
We  first cast Feller's argument in such terms before proceeding to a more general discussion.


\subsection{Scope of Paper}
\label{sec:scope}

The Mittag-Leffler function is of growing interest in probability theory and physics, 
with a diversity of applications, notably fractional calculus.
A comprehensive study of the  properties and applications of  the Mittag-Leffler function 
and its numerous generalisations is beyond the scope of this paper.
We  consciously restrict the scope  to  the theme of  complete monotonicity and  Mittag-Leffler functions, 
underpinned by 
Bayesian reasoning.

Other studies that explicitly discuss complete monotonicity and  Mittag-Leffler functions 
build upon  complex analytic approaches similar to Pollard's rather than
the probabilistic  underpinning discussed here.
For example, deOliveira {\it et al.}~\cite{Oliveira}  and Mainardi and Garrappa~\cite{MainardiGarrappa} studied the
complete monotonicity of $x^{\beta-1}E_{\alpha,\beta}^\gamma(-x^\alpha)$, whereas 
G\'{o}rska {\it et al.}~\cite{Gorska} explored  the complete monotonicity of $E_{\alpha,\beta}^\gamma(-x)$. 
$E_{\alpha,\beta}^\gamma(x)$ is the three-parameter variant of the Mittag-Leffler function, 
also  known as the Prabhakar function.
These papers comment on the fundamental  importance of the complete monotonicity  of Mittag-Leffler functions used in 
the modelling of physical phenomena, such as anomalous dielectric relaxation and viscoelasticity.



Finally, we are keenly aware that there are other views on the interpretation of ``methods of probability theory''.  
We comment on this before discussing the Bayesian approach in detail.

\subsection{Probabilistic Perspectives}
\label{sec:perspectives}

The phrase  `methods of probability theory' used by Pollard may suggest an 
 experiment  with random outcomes as a  fundamental  metaphor. 
Indeed, Pollard's $P_\alpha$, which is referred to as  the Mittag-Leffler distribution in the probabilistic literature,
 is derived  as a limiting distribution of 
a  P{\' o}lya urn scheme  ({\it e.g.}\ Janson~\cite{Janson}).

Diversity  of approach is commonplace in  probability theory and mathematics more generally.
For example, in a  context of nonparametric  Bayesian analysis, 
 Ferguson~\cite{Ferguson1} constructed the Dirichlet process based on 
the  gamma distribution as the fundamental probabilistic concept, 
without invoking a random experiment.
Blackwell and MacQueen~\cite{BlackwellMacQueen}  observed that the Ferguson approach
``involves a rather deep study of the gamma process'' as they proceeded to give  an alternate construction 
based on the metaphor of a  generalised P{\' o}lya urn scheme.
Adopting the one approach is not to deny or diminish the other,  but to bring attention to the diversity of thinking in  probability theory,
 even when the end result is the same mathematical  object.
We look upon this as healthy complementarity rather than undesirable contestation.
 

We   discuss  complete  monotonicity by methods of probability theory in the sense of Bayesian reasoning.
For the purpose at hand, we have no need to invoke an underlying random experiment or indeed  an explicit random variable,
while not denying the latter  as an alternative probabilistic approach.
Hence, for example, we  shall continue to express the Laplace transform of a distribution as an explicit integral 
rather than as an expectation $\mathbb{E}\left[e^{-sX}\right]$ 
for  a random variable $X$.

\section{A Bayesian Approach}
\label{sec:Bayesian}

First, we note that the scale change $s\to t^{1/\alpha}s$ $(t>0)$ in~(\ref{eq:stable})  gives
\begin{align}
 e^{-t s^\alpha}  &= \int_0^\infty e^{-s x} f_\alpha(x\,t^{-1/\alpha}) t^{-1/\alpha}\,dx 
 \equiv \int_0^\infty e^{-s x} f_\alpha(x \vert t) \,dx
\label{eq:stablescaled}
\end{align}
where $f_\alpha(x\vert t)\equiv f_\alpha(x\,t^{-1/\alpha}) t^{-1/\alpha}$ is 
the stable density conditioned on the scale parameter $t$,
with $f_\alpha(x) \equiv f_\alpha(x\vert 1)$.
Correspondingly, the  stable distribution conditioned  on $t$ is
\begin{align}
F_\alpha(x\vert t) &= \int_0^x\, f_\alpha(u\vert t)\,du
     =  \int_0^{x t^{-1/\alpha}}  f_\alpha(u) \,du  \equiv F_\alpha(x t^{-1/\alpha})
\label{eq:stablescaleddistribution}
\end{align}
with Laplace transform 
$ e^{-t s^\alpha}/s$.

We then assign a distribution $G(t)$  to the scale parameter $t$ of 
$F_\alpha(x\vert t)$.
Then,  by the sum and  product rules of probability theory, the  unconditional  or marginal distribution  $M_\alpha(x)$  over $x$ is
\begin{align}
M_\alpha(x) &= \int_0^\infty F_\alpha(x\vert t) dG(t)
\label{eq:marginaldistribution}
\intertext{with Laplace transform}
\int_0^\infty e^{-s x} M_\alpha(x) \,dx  
       &= \frac{1}{s} \int_0^\infty e^{-ts^\alpha}\, dG(t)
\label{eq:marginaldistributionLaplace}
\end{align}
$M_\alpha$ 
is also referred to as  a  mixture distribution, 
arising from randomising or mixing the parameter $t$ in $F_\alpha(x\vert t)$  with $G(t)$.
This has the same import as saying that we assign a prior distribution  $G(t)$ on $t$ 
and we shall continue to use the latter language.

 $G$ may depend on one or more parameters. 
A notable example is the gamma distribution $G(\mu,\lambda)$ with shape and scale parameters  $\mu>0, \lambda>0$ respectively:
\begin{align}
dG(t\vert \mu,\lambda) 
 &=\dfrac{\lambda^\mu}{\Gamma(\mu)}\,t^{\mu-1}e^{-\lambda t}\, dt 
\label{eq:gammadistribution} 
\end{align}
$\lambda$ is not fundamental and may  be set to  $\lambda=1$ by  change of scale 
$t\to\lambda t$, while $\mu$ controls the shape of $G(t\vert \mu,\lambda)$.
The marginal~(\ref{eq:marginaldistribution})  becomes $M_{\alpha}(x\vert \mu,\lambda)$, with Laplace transform
\begin{align}
\int_0^\infty e^{-s x} M_\alpha(x\vert \mu,\lambda) \,dx &=  \frac{1}{s}\left(\frac{\lambda}{\lambda+ s^\alpha}\right)^\mu
       = \frac{1}{s}\left(1-\frac{s^\alpha}{\lambda+ s^\alpha}\right)^\mu
\label{eq:marginaldistributionLaplace}
\end{align}
We may now state Feller's approach  from a Bayesian  perspective.

\subsection{A Bayesian View of Feller's Approach}
\label{sec:BayesianFeller}

The case $\mu=1$ in~(\ref{eq:gammadistribution}) gives the exponential distribution $dG(t\vert \lambda) =  \lambda e^{-\lambda t}dt$.
Then $M_\alpha(x\vert \lambda)\equiv M_\alpha(x\vert \mu=1,\lambda)$ is 
\begin{align}
M_\alpha(x\vert \lambda) &= \int_0^\infty F_\alpha(x\vert t) dG(t\vert \lambda)
    = \lambda \int_0^\infty F_\alpha(x\vert t) e^{-\lambda t}\,dt
\label{eq:expprior} 
\end{align}
The Laplace transform of $M_\alpha(x\vert \lambda)$, read from~(\ref{eq:marginaldistributionLaplace}) with $\mu=1$, is
\begin{alignat}{3}
\int_0^\infty e^{-s x} &M_\alpha(x\vert \lambda) \,dx 
&&= \frac{1}{s}- \frac{s^{\alpha-1}}{\lambda+s^\alpha}
\label{eq:LaplaceMarginaldistribution} \\
\implies 
&\quad M_\alpha(x\vert \lambda) &&= 1-E_\alpha(-\lambda x^\alpha) \\
\implies 
&E_\alpha(-\lambda x^\alpha) &&= 1-M_\alpha(x\vert \lambda)
  = \lambda \int_0^\infty (1-F_\alpha(x\vert t)) e^{-\lambda t} \,dt
\end{alignat}
This reproduces Feller's result~(\ref{eq:FellerPollard1}) 
from a Bayesian perspective.
The difference is purely a matter  of conceptual outlook: 
\begin{description}[leftmargin=3.2em,font=\rm\sffamily]
\item[Feller:] Study  the  two-dimensional Laplace transform of the bivariate distribution $1-F_\alpha(x t^{-1/\alpha})$, 
where $F_\alpha$ is the stable distribution. 
Deduce that $E_\alpha(-\lambda x^\alpha)/\lambda$ is the Laplace  transform of $1-F_\alpha(x t^{-1/\alpha})$ over $t$ at fixed $x$,
where $\lambda$ is the Laplace variable.
\item[Bayes:] Assign an exponential prior distribution $G(t\vert 1,\lambda)$ to the scale factor $t$ of 
$F_\alpha(x\vert t)\equiv F_\alpha(x t^{-1/\alpha})$, where $G(t\vert\mu,\lambda)$  is the gamma distribution.
Marginalise 
over $t$ to generate the Feller result directly.
\end{description}
Feller himself might also have established the result by the latter reasoning. 
Under subordination of processes~\cite{Feller2}(p451),
he  discussed   mixture distributions but he did not  specifically discuss the Mittag-Leffer function in this context in his published work.
 The task fell on Pillai~\cite{Pillai} 
 to study $M_\alpha(x\vert\mu)\equiv M_\alpha(x\vert \mu, \lambda=1)$,
 including its infinite divisibility and the corresponding  Mittag-Leffer  stochastic process. 
 He also proved that  $M_\alpha(x\vert 1)=1-E_\alpha(-x^\alpha)$ (as  discussed  above), 
which he referred to as the Mittag-Leffer distribution.
There are thus two distributions bearing the  name ``Mittag-Leffer distribution'': 
$M_\alpha(x)=1-E_\alpha(-x^\alpha)$ and $P_\alpha(t) =1-  F_\alpha(t^{-1/\alpha})$.
We shall use the term to refer to the latter distribution  in the balance of our discussion.

\subsection{A Bayesian Generalisation}
\label{general}

The natural question arising from the Bayesian approach is whether there might be other choices of $\mu$ 
in $G(\mu,\lambda)$ (or indeed other choices of $G$ altogether) 
that yield the Pollard result and, if so,  what   insight they might  offer.
At face value, there would appear to be nothing further to be said since other choices of $\mu$ 
can be expected to lead to different results,  beyond  the study of the Mittag-Leffler function. 

While~(\ref{eq:gammadistribution}) is not defined for $\mu=0$, we note 
that (with $\mu\Gamma(\mu)=\Gamma(\mu+1)$):
\begin{align}
\frac{1}{\mu}dG(t\vert \mu,\lambda)
 &=\dfrac{\lambda^\mu}{\Gamma(\mu+1)}\,t^{\mu-1}e^{-\lambda t}\, dt  \\
 \implies
\lim_{\mu\to0} \frac{1}{\mu}dG(t\vert \mu,\lambda)
 &= t^{-1}e^{-\lambda t}\, dt 
\end{align}
Against our own expectation, we have discovered that  this ``$\mu=0$'' case 
({\it i.e.}\ the unnormalised distribution with density $t^{-1}e^{-\lambda t}$) 
generates  a novel integral representation of the Mittag-Leffler function.
This is the main result of this paper, which we state next.
We follow with a discussion of the  Bayesian reasoning that led to the discovery and the generalisation that  arises from that.

\section{Main Contribution}
\label{sec:contribution}

\begin{proposition}
The Mittag-Leffler function $E_\alpha(-\lambda x^\alpha)$ $(x\ge0, \lambda>0)$ has the integral representation 
\begin{align}
\alpha\,E_\alpha(-\lambda x^\alpha) &= x \int_0^\infty f_\alpha(x\vert t) \, t^{-1} e^{-\lambda t} \,dt  \qquad 0<\alpha<1
\label{eq:ML_intrep0}
\intertext{where $f_\alpha(x\vert t)$ is the stable density with Laplace  transform $e^{-t s^\alpha}$.
This leads to the Pollard result}
E_\alpha(-x) &= \frac{1}{\alpha} \int_0^\infty  f_\alpha(u^{-1/\alpha})\,  u^{-1/\alpha-1} e^{-xu} \, du
\label{eq:PollardML_intrep0}
\end{align} 
Thus $E_\alpha(-x)$  is completely monotone.
\label{prop:main}
\end{proposition}

\begin{proof}[Proof of Proposition \ref{prop:main}]
The Laplace transform of the RHS of~(\ref{eq:ML_intrep0}) is
\begin{align*}
\int_0^\infty e^{-sx} \,x &\int_0^\infty f_\alpha(x\vert t) \, t^{-1} e^{-\lambda t} \,dt \, dx\\
&= -\frac{d}{ds} \int_0^\infty t^{-1} e^{-\lambda t} \int_0^\infty  e^{-sx}  f_\alpha(x\vert t) \, dx  \,dt   \\
&= -\frac{d}{ds} \int_0^\infty   t^{-1} e^{-\lambda t} \, e^{-t s^\alpha} \,dt    \\
&= \alpha s^{\alpha-1}  \int_0^\infty e^{-(\lambda + s^\alpha)t}  \,dt   \\
&= \alpha \frac{s^{\alpha-1}}{\lambda+s^\alpha} 
\end{align*}
which is the Laplace transform of the LHS $\alpha E_\alpha(-\lambda x^\alpha)$ of~(\ref{eq:ML_intrep0}). 
The expression~(\ref{eq:PollardML_intrep0}) for $E_\alpha(-x)$ follows from simple substitution 
\begin{align*}
\alpha\,E_\alpha(-x) &= x^{1/\alpha} \int_0^\infty f_\alpha(x^{1/\alpha}\vert t) \, t^{-1} e^{-t} \,dt   \\
                   &= x^{1/\alpha} \int_0^\infty f_\alpha(x^{1/\alpha}t^{-1/\alpha}) \, t^{-1/\alpha-1} e^{-t} \,dt \\
                   &= \int_0^\infty f_\alpha(u^{-1/\alpha})\,  u^{-1/\alpha-1} e^{-xu} \, du
\end{align*}
which is the Pollard result.
\end{proof}
Since Pollard's result  follows almost trivially from~(\ref{eq:ML_intrep0}),
the pertinent question  is where does this integral representation come from in the first place?
Once again, if we were to accept  it at face value as, perhaps, a fortunate guess (it does not take a particularly subtle form after all)
then there would be nothing further to be said.

Pursuing further, we observe that $t^{-1}e^{-\lambda t}$ is the L\'{e}vy density of the infinitely divisible gamma distribution.  
There is indeed an   intimate relationship between completely monotone functions and the theory of 
infinitely divisible  distributions on the nonnegative half-line $\mathbb{R}_{+}=[0,\infty)$.
This is a topic well-studied by Feller~\cite{Feller2}. 
In the balance of this paper, we shall turn  to this topic.
But first, we discuss a generalisation of the Pollard result that follows from a variant of~(\ref{eq:ML_intrep0}).

\section{Generalisation of the Pollard Result}
\label{sec:tilting}

As mentioned in Section~\ref{sec:perspectives}, $P_\alpha$ of~(\ref{eq:PollardP})
is known as  the Mittag-Leffler distribution in probabilistic literature.
There is a two-parameter generalisation known as the  generalised Mittag-Leffler distribution
 $P_{\alpha,\theta}$ 
(Pitman~\cite{Pitman_CSP}, p70 (3.27)),  also   denoted by ${\rm ML(\alpha,\theta})$
(Goldschmidt and Haas~\cite{GoldschmidtHaas}, Ho {\it et al.}~\cite{HoJamesLau}).
It may be written as  
\begin{align}
P_{\alpha,\theta}(t) 
  &= \frac{\Gamma(\theta+1)}{\Gamma(\theta/\alpha+1)} \int_0^t u^{\theta/\alpha}\, dP_\alpha(u)
\qquad \theta>-\alpha
\label{eq:genMLdistribution} 
\end{align}
$P_\alpha$ is the $\theta=0$ case: $P_\alpha\equiv P_{\alpha,0}$.
Janson~\cite{Janson} showed that $P_{\alpha,\theta}$  may be  constructed  
as a limiting distribution of a   P{\' o}lya urn scheme. 
In a concise description  of this construction, 
Goldschmidt and Haas~\cite{GoldschmidtHaas}  observed that 
``generalised Mittag-Leffler distributions arise naturally in the context of urn models''.
We show that the generalised Mittag-Leffler distribution and its Laplace transform also arise naturally 
in the context of Bayesian reasoning.

Given the  stable density $f_\alpha(x)$ ($0<\alpha<1$),  
the two-parameter density $f_{\alpha,\theta} (x)\propto  x^{-\theta}f_\alpha(x)$ 
for some parameter $\theta$ (range discussed below) is said to be a `polynomially  tilted' variant of $f_\alpha(x)$
({\it e.g.}\ Arbel {\it et al.}~\cite{Arbel}, Devroye~\cite{Devroye}, James~\cite{James_Lamperti}).
Inspired by this, we consider the polynomially tilted  density 
 $f_{\alpha,\theta} (x\vert t)\propto x^{-\theta}f_\alpha (x\vert t)$
conditioned on a general scale factor $t>0$, 
where  $f_\alpha(x)\equiv f_\alpha(x\vert t=1)$.
The normalised tilted   density conditioned on $t$ is
\begin{align}
 f_{\alpha,\theta}(x\vert t)= C_{\alpha,\theta} (t) \,x^{-\theta}f_\alpha (x\vert t)
 \quad{\rm where}\quad
C_{\alpha,\theta}(t)  &= \frac{\Gamma(\theta+1)}{\Gamma(\theta/\alpha+1)} t^{\theta/\alpha} 
\label{eq:tiltedstablenormfactor}
\end{align}
so that $f_{\alpha,\theta}(x\vert t)$ is defined for $\theta/\alpha+1>0$, or $\theta>-\alpha$.

The key idea here is to tilt the conditional stable density $f_\alpha(x\vert t)$ and assign a prior distribution to $t$ rather than
merely tilt $f_\alpha(x)\equiv f_\alpha(x\vert t=1)$. 
We consider  a variant of~(\ref{eq:ML_intrep0}) with the prior distribution $t^{-1}e^{-\lambda t}dt$ 
and $f_{\alpha}(x\vert t)$ replaced by  $f_{\alpha,\theta} (x\vert t)$.
This induces  a corresponding  two-parameter function  $h_{\alpha,\theta} (x\vert\lambda)$ in place of $E_\alpha(-\lambda x^\alpha)$,
for which the following holds: 
\begin{proposition}
Let $h_{\alpha,\theta} (x\vert\lambda)$ be defined by 
\begin{align}
 \alpha\, h_{\alpha,\theta} (x\vert\lambda) &= x \int_0^\infty f_{\alpha,\theta} (x\vert t) \, t^{-1} e^{-\lambda t} \,dt 
\label{eq:genMLfunction1} \\
   &= \frac{\Gamma(\theta+1)}{\Gamma(\theta/\alpha+1)} \, x^{1-\theta}
    \int_0^\infty f_{\alpha}(x\vert t)\, t^{\theta/\alpha-1}\, e^{-\lambda t} \, dt
\label{eq:genMLfunction2} 
\intertext{Then $h_{\alpha,\theta} (x^{1/\alpha})\equiv h_{\alpha,\theta} (x^{1/\alpha}\vert\lambda=1)$  is completely monotone with}
 h_{\alpha,\theta} (x^{1/\alpha})  &=    \int_0^\infty e^{-xt}dP_{\alpha,\theta}(t) 
\label{eq:LTgenML} 
\end{align}
\label{prop:genML}
\end{proposition}
\begin{proof}[Proof of Proposition \ref{prop:genML}]
\begin{align}
 \alpha\, h_{\alpha,\theta} (x^{1/\alpha})  
   &= \frac{\Gamma(\theta+1)}{\Gamma(\theta/\alpha+1)} \, x^{1/\alpha-\theta/\alpha}
    \int_0^\infty f_{\alpha}(x^{1/\alpha}\vert t)\, t^{\theta/\alpha-1}\, e^{-t} \, dt \nonumber \\
   &= \frac{\Gamma(\theta+1)}{\Gamma(\theta/\alpha+1)} \, x^{1/\alpha-\theta/\alpha}
    \int_0^\infty f_\alpha(x^{1/\alpha}t^{-1/\alpha}) \, t^{\theta/\alpha-1/\alpha-1}\, e^{-t} \, dt \nonumber \\
\intertext{The change of variable $t\to xt$ leads to}    
 h_{\alpha,\theta} (x^{1/\alpha})  
   &= \frac{\Gamma(\theta+1)}{\Gamma(\theta/\alpha+1)}
    \int_0^\infty e^{-xt}\,  t^{\theta/\alpha} \left(\frac{1}{\alpha}\, f_\alpha(t^{-1/\alpha}) \, t^{-1/\alpha-1}\right) \, dt  \nonumber \\
   &= \frac{\Gamma(\theta+1)}{\Gamma(\theta/\alpha+1)}
    \int_0^\infty e^{-xt} \,  t^{\theta/\alpha} \, dP_\alpha(t)  \nonumber \\
    &= \int_0^\infty e^{-xt} \, dP_{\alpha,\theta}(t) \nonumber
\end{align}
Hence $h_{\alpha,\theta} (x^{1/\alpha})$  is completely monotone.
This generalises the Pollard result, which  is the particular case $\theta=0$: 
$P_{\alpha,0}(t)=P_\alpha(t) \implies
 h_{\alpha,0}(x^{1/\alpha})=E_\alpha(-x)$.
\end{proof}
It would be consistent with the foregoing discussion  to refer to $h_{\alpha,\theta} (x^{1/\alpha})$, 
as the generalised (two-parameter) Mittag-Leffler function.
However, 
we note that there already exists a two-parameter  
generalised Mittag-Leffler function defined by
\begin{align*}
E_{\alpha,\beta}(x) &= \sum_{k=0}^\infty \frac{x^k}{\Gamma(\alpha k+\beta)}
\label{eq:ML2}
\end{align*}
where $E_{\alpha} (x)\equiv E_{\alpha,1}(x)$.

This completes the discussion of the primary contribution of this paper --  
a Bayesian perspective on the complete monotonicity of the Mittag-Leffler function and its generalisation.
We now turn to the topic of  infinitely divisible distributions on $\mathbb{R}_{+}$, from which 
Proposition~\ref{prop:main} arises.

\section{Infinitely Divisible Distributions on $\mathbb{R}_{+}$}
\label{sec:ID}

Infinitely divisible distributions on $\mathbb{R}_{+}$ are covered in Feller~\cite{Feller2} (XIII.4,~XIII.7) 
as well as Steutel and van Harn (SvH)~\cite{SteutelvanHarn} (III).
The topic has an intimate relationship with completely monotone functions, as discussed in both texts.
Another relevant text in this context is Schilling {\it et al.}~\cite{Schilling} on Bernstein functions.
Sato~\cite{Sato}  considers infinitely divisible distributions on $\mathbb{R}^d$,
but the deliberate restriction to $\mathbb{R}_{+}$ makes for simpler discussion 
and relates directly to the core concept of complete monotonicity that is of interest here.
Nonetheless, we shall refer to Sato  as appropriate.

A  contribution of this paper is  the representation of infinite divisibility as a commutative diagram,
which readily leads to a limit relation enabling the  direct generation of a L\'{e}vy measure  
from its associated infinitely divisible distribution. 
Although the diagrammatic motivation is new,
the limit relation is mentioned in Steutel and van Harn~\cite{SteutelvanHarn}~(III)
and it is conceptually equivalent to that due to  Sato~\cite{Sato}. 
It is the  limit relation that   leads to the 
representation~(\ref{eq:ML_intrep0}) 
from which the Pollard result  follows.


\begin{definition}
A probability  distribution with Laplace transform $\varphi$  is infinitely divisible  {\rm (ID)}
iff for  $n>0$, the positive $n^{\rm th}$ root $\varphi^{1/n}$  is  also the Laplace transform of a probability  distribution.
\label{def:ID}
\end{definition}
\begin{theorem}[Feller~\cite{Feller2},  XIII.7, p450]
The function $\varphi$ is the Laplace transform of an  infinitely divisible probability  distribution 
iff it takes the form $\varphi=e^{-\psi}$ 
where $\psi$ has a completely monotone derivative $\psi\,^\prime$ 
and $\psi(0)=0$. 
\label{thm:ID}
\end{theorem}
\begin{proof}[Proof of Theorem~\ref{thm:ID}]
See Feller~\cite{Feller2}, XIII.7, p450.
\label{proof:ID}
\end{proof}
Infinite divisibility on $\mathbb{R}_{+}$ thus amounts to the study of the function $\psi$.
We introduce a scale parameter $\mu>0$ so that $\varphi(s)\to\varphi(s\vert \mu)=e^{-\mu\psi}$
is the   Laplace transform of an  infinitely divisible probability  
density $f(x\vert \mu$).
Correspondingly, $\varphi(s\vert \mu)^{1/n}=e^{-\mu\psi(s)/n}=\varphi(s\vert \tfrac{\mu}{n})$
is the Laplace transform of 
$f(x\vert \tfrac{\mu}{n})$.
Since $\psi\,^\prime$  is completely monotone,  it is the  Laplace transform of a 
density $r(x)$, say (which need not be normalisable -- 
{\it i.e.}\ unlike $\psi(0)$, $\psi\,^\prime(0)$ need not be finite).
$\ell(x)=r(x)/x$ is known as  the density of the L\'{e}vy measure $L(dx)=\ell(x)dx$, or simply  the L\'{e}vy density.
We shall also refer to the measures  $R(dx)=r(x)dx=xL(dx)$ and $F(dx\vert \mu) = f(x\vert \mu)dx$.


We find it helpful to represent ID objects and relationships amongst them as a commutative diagram.
If we seek the  density $r(x)$ of  a given ID density $f(x\vert \mu)$, we
may proceed as illustrated in the upper diagram of (\ref{eq:LKCD}) (where $\scrL$ denotes the Laplace transform):
\begin{enumerate}
\item take the Laplace  transform $\varphi(s\vert \mu)=e^{-\mu\psi(s)}$ of $f(x\vert \mu)$
\item take (minus) the logarithmic derivative of $\varphi(s\vert \mu)$ to obtain  $\mu\,\psi\,^\prime(s)$
\item evaluate the inverse Laplace  transform of $\psi\,^\prime(s)$ to obtain $r(x)=x\,\ell(x)$
\end{enumerate}

\begin{equation}
\begin{tikzpicture}[auto,scale=1.5, baseline=(current  bounding  box.center)]
\newcommand*{\size}{\scriptsize}%
\newcommand*{\gap}{.2ex}%
\newcommand*{\width}{2.5}%
\newcommand*{\height}{1.5}%

\node (P) at (0,0)  {$f(x\vert \mu)$};
\node (Q) at ($(P)+(\width,0)$) {$\varphi(s\vert \mu) = e^{-\mu\psi(s)}$};
\node (B) at ($(P)-(0,\height)$) {$\mu r(x)$};
\node (C) at ($(B)+(\width,0)$) {$\mu\psi\,^\prime(s)$};   

\draw[Myarrow] ([yshift =  \gap]P.east)  --  node[above] {\size $\scrL$} ([yshift = \gap]Q.west) ;
\draw[Myarrow]([xshift  =  \gap]Q.south) --  
node [right] {\size $-\frac{\varphi^\prime(s\vert \mu)}{\varphi(s\vert \mu)}$} ([xshift =  \gap]C.north);
\draw[Myarrow,dashed] ([xshift =  \gap]P.south) --
node [left] {\size ?}  ([xshift =  \gap]B.north);
\draw[Myarrow] ([yshift = -\gap]C.west) --  node[above] {\size ${\scrL}^{-1}$}  ([yshift = -\gap]B.east); 

\node (P) at (0,-1.5*\height)  {$f(x\vert \mu)$};
\node (Q) at ($(P)+(\width,0)$) {$\varphi(s\vert \mu)$};
\node (B) at ($(P)-(0,\height)$) {$\mu r(x)$};
\node (C) at ($(B)+(\width,0)$) {$\mu\psi \,^\prime(s)$};   

\draw[Myarrow] ([yshift =  \gap]P.east)  --  node[above] {\size $\scrL$} ([yshift = \gap]Q.west) ;
\draw[Myarrow]([xshift  =  \gap]Q.south) --  
node[right] {\size $-\displaystyle\lim_{n\to\infty}n\varphi\,^\prime(s\vert \tfrac{\mu}{n})$} ([xshift =  \gap]C.north);
\draw[Myarrow] ([xshift =  \gap]P.south) -- 
node[left] {\size ${{\displaystyle\lim_{n\to\infty}}n x f(x\vert \frac{\mu}{n})}$} ([xshift =  \gap]B.north);
\draw[Myarrow] ([yshift = -\gap]C.west) --  node[above] {\size ${\scrL}^{-1}$}  ([yshift = -\gap]B.east); 
\end{tikzpicture}
\label{eq:LKCD}
\end{equation}
The natural question suggested by the upper diagram is whether we can find an equivalent direct transition from 
 $f(x\vert \mu)$ to $\mu r(x)$. 
 The answer is affirmative, as formalised in the following:
 \begin{corollary}
 Let $F(dx\vert \mu) = f(x\vert \mu)dx$ be infinitely divisible with  L\'{e}vy measure $L(dx)=R(dx)/x$ 
 with density $\ell(x)=r(x)/x$.
 Then the following 
 holds:
 \label{cor:limit}
 \begin{alignat}{4}
\mu\,\ell(x) &= \lim_{n\to\infty}n f(x\vert \tfrac{\mu}{n}) \quad &&{\rm or}\quad
\mu L(dx) &&= \lim_{n\to\infty}n F(dx\vert \tfrac{\mu}{n}) 
\label{eq:limitL} \\
\mu\,r(x)   &= \lim_{n\to\infty}n x f(x\vert \tfrac{\mu}{n}) \quad &&{\rm or}\quad
\mu R(dx) &&= \lim_{n\to\infty}n x F(dx\vert \tfrac{\mu}{n})
\label{eq:limitR}
\end{alignat}
 \end{corollary}
 \begin{proof}[Proof of Corollary \ref{cor:limit}]
By Theorem~\ref{thm:ID}, $f(x\vert \mu)$ has Laplace transform  
$\varphi(s\vert \mu)=e^{-\mu\psi(s)}$ so that  $\varphi(s\vert \mu=0)=1$. Hence 
 \begin{align}
 -\lim_{n\to\infty}n\varphi\,^\prime\left(s\vert \tfrac{\mu}{n}\right) &= 
  \lim_{n\to\infty}\mu\psi \,^\prime(s) \varphi\left(s\vert \tfrac{\mu}{n}\right)  
  = \mu\psi \,^\prime(s)\varphi(s\vert 0) = \mu\psi \,^\prime(s) 
\label{eq:limitLaplace}
\end{align}
Since $-\varphi\,^\prime(s\vert \mu)$ is the Laplace  transform of $xf(x\vert \mu)$, it follows that
\begin{align}
\mu\,r(x) &= \lim_{n\to\infty}n\, x\,f\left(x\vert \tfrac{\mu}{n}\right) \quad\implies\quad
\mu\,\ell(x) = \lim_{n\to\infty}n \,f\left(x\vert \tfrac{\mu}{n}\right)
\label{eq:limit} 
\end{align}
This  is   invariant under scaling by $C>0$: $f(x\vert \mu)\to Cf(x\vert \mu)$. 
 \label{proof:limit}
 \end{proof}
 The  lower diagram of~(\ref{eq:LKCD}) is the desired  commutative  diagram.
To be clear, Corollary~\ref{cor:limit} is known.
To aid comparison with the literature,
Corollary~\ref{cor:limit} implies that, given a function $h$ on $\mathbb{R}_{+}$ (and  finite $x$) 
 \begin{align}
\int_0^x h(u) R(du)  &= \lim_{n\to\infty} \tfrac{n}{\mu} \int_0^x u \, h(u) F(du\vert \tfrac{\mu}{n})
\label{eq:intlimr} \\
\int_0^x h(u) L(du)  &= \lim_{n\to\infty} \tfrac{n}{\mu} \int_0^x \, h(u) F(du\vert \tfrac{\mu}{n})
\label{eq:intliml} 
\end{align}
The relation  in SvH~\cite{SteutelvanHarn}~(III(4.7))  is a particular case of~(\ref{eq:intlimr}).
Also,  Sato~\cite{Sato} (Corollary~8.9) proved the limit relation 
 \begin{align}
\int_{\mathbb{R}^d} h(x) L(dx)  &= \lim_{t\to0} t^{-1} \int_{\mathbb{R}^d}  h(x) F(dx\vert t)
\label{eq:sato} 
\end{align}
for suitably behaved $h(x)$.
Choosing   $\mathbb{R}_{+}$ instead of $\mathbb{R}^d$ in Sato's relation~(\ref{eq:sato})
and setting   $t=\mu/n$  reproduces~(\ref{eq:intliml}) where $x=\infty$ is allowable.
However, working in $\mathbb{R}_{+}$ from the outset   makes for much  simpler discussion  of infinitely divisible distributions on 
$\mathbb{R}_{+}$ relative to working  in $\mathbb{R}^d$ and then trying to infer  behaviour on $\mathbb{R}_{+}$ as a special case.

The contribution here is the intuitive  manner in which the limit relation  arises from a commutative diagram argument. 
Furthermore, the limiting rule from $f(x\vert \mu)$ to $\mu r(x)$ stated here explicitly  preserves the scale factor $\mu$,
in keeping with the indirect route via the Laplace transform, which does not actually require the evaluation of a limit.

Aside from SvH and Sato, there  appears to be limited discussion of 
inferring the L\'{e}vy measure or its properties directly from the corresponding infinitely divisible distribution. 
For instance, Barndorff-Nielsen and Hubalek~\cite{BarndorffHubalek} cited Sato's relation 
at the start before turning  to 
``the opposite problem, that 
 of calculating $F(dx\vert t)$ from $L(dx)$''.



Since the direct  route of Corollary~\ref{cor:limit} and the indirect Laplace route both lead from $f(x\vert \mu)$
 to the same object $\mu\,r(x)$, 
the  natural question  is whether Corollary~\ref{cor:limit}  is of  much practical value. 
The  answer is that  the two routes can lead to different  representations of the same object  $r(x)$.
Therein lies the practical value of Corollary~\ref{cor:limit}. 
To that end, we first introduce some additional properties of complete monotonicity that we shall need.

 \section{More on Complete Monotonicity}
 \label{sec:CM}
 We summarise additional  properties of  completely monotone functions, as covered in Feller~\cite{Feller2}, XIII,
 except for the proofs given here.
\begin{proposition}
If $\varphi$ and $\vartheta$ are  completely monotone,   so is their product $\varphi\vartheta$
\label{prop:CMproduct}
\end{proposition}
\begin{proof}[Proof of Proposition \ref{prop:CMproduct}]
See Feller~\cite{Feller2}, XIII.4, p441.\\
Alternatively, being completely monotone, $\varphi$ and $\vartheta$ are  Laplace transforms of densities. 
The product $\varphi\vartheta$ is thus the Laplace transform of the convolution of said densities, which
is also a density. Therefore $\varphi\vartheta$ is completely monotone.
\label{proof:CMproduct}
\end{proof}

\begin{proposition}
If $\varphi$ is completely monotone and $\eta$ is a positive function with a completely monotone derivative,  
$\varphi(\eta)$ is completely monotone
\label{prop:composition}
\end{proposition}
\begin{proof}[Proof of Proposition \ref{prop:composition}]
If $\varphi(s)$ is completely monotone,  so is $-\varphi\,^\prime(s)$. 
Consider $\varphi(\eta)$ where $\eta(s) > 0$ and $\eta\,^\prime(s)$ is completely monotone. 
Of necessity, $\varphi(\eta)>0$ and 
\begin{align}
-\varphi\,^\prime(\eta) &= \left(-\frac{d\varphi(\eta)}{d\eta}\right)\eta\,^\prime(s)
\label{eq:composition}
\end{align}
The RHS is a product of two completely monotone functions. 
Therefore, by Proposition~\ref{prop:CMproduct}, $-\varphi\,^\prime(\eta)$ is completely monotone.  
This, along with 
$\varphi(\eta)>0$, completes the proof that $\varphi(\eta)$ is completely monotone.
\label{proof:composition}
\end{proof}

\begin{proposition}
If $\varphi$ is the Laplace  transform of an infinitely divisible distribution and $\eta$ is a positive function with a completely monotone derivative,  $\varphi(\eta)$ is also the Laplace  transform of  an infinitely divisible distribution.
\label{prop:IDcomposition}
\end{proposition}
\begin{proof}[Proof of Proposition \ref{prop:IDcomposition}]
By Theorem~\ref{thm:ID}, $\varphi$ is the Laplace  transform of an infinitely divisible distribution iff 
$\varphi=e^{-\psi}$ where $\psi\,^\prime$ is completely monotone.
By Proposition~\ref{prop:composition}, if $\psi^\prime$ is completely monotone and  
$\eta$ is a positive function with a completely monotone derivative, $\psi^\prime(\eta)$ is completely monotone.
Hence, for such $\eta$,  $\varphi(\eta)=e^{-\psi(\eta)}$ is  the Laplace  transform of an infinitely divisible distribution.
\label{proof:IDcomposition}
\end{proof}

We may now revisit the Bayesian formulation of  Section~\ref{sec:Bayesian}.

\section{Bayesian Approach  Revisited}
\label{sec:BayesianID}


\begin{theorem}
Let $m(x\vert \mu)$, $f(x\vert y)$ and $g(y\vert \mu)$ be densities on $[0,\infty)$ such that 
\begin{align}
m(x\vert \mu) &= \int_0^{\infty}f(x\vert y)\,g(y\vert \mu)\,dy
\label{eq:mix}
\end{align}
If $f(x\vert y)$ and $g(y\vert \mu)$ are infinitely divisible, with Laplace transforms 
$e^{-y\eta(s)}$ and $e^{-\mu\psi(s)}$ respectively, where $\eta\,^\prime(s)$ and $\psi\,^\prime(s)$ are  completely monotone, then:
\begin{enumerate}
\item  $m(x\vert \mu)$ is also infinitely divisible 
\item  the L\'{e}vy density $\xi(x)$, say, of $m(x\vert \mu)$  is 
\begin{align}
\xi(x) &= \int_0^\infty f(x\vert y)\,\ell(y)\,dy
\label{eq:marginalLevy}
\end{align}
\end{enumerate}
where $\ell(y)$ is the L\'{e}vy density of $g(y\vert \mu)$.
\label{thm:marginal}
\end{theorem}
Feller~\cite{Feller2} (p451) discussed 
the first part of this theorem  in an example on subordination of processes. 
One may also refer to SvH~\cite{SteutelvanHarn}~VI(Proposition~2.1).
The additional  contribution here is the integral  representation of the L\'{e}vy density of $m(x\vert \mu)$ in the second part
of the theorem.

\begin{proof}[Proof of Theorem~\ref{thm:marginal}] 
Since $f(x\vert y)$ and $g(y\vert \mu)$ are infinitely divisible, by Theorem~\ref{thm:ID}, their  Laplace transforms take the form 
$e^{-y\eta(s)}$ and $e^{-\mu\psi(s)}$ respectively, where $\eta\,^\prime(s)$ and $\psi\,^\prime(s)$ are  completely monotone.
In turn, $e^{-\mu\psi(\eta(s))}$ is the  Laplace transform of $m(x\vert \mu)$ induced by~(\ref{eq:mix}).
Hence, by Proposition~\ref{prop:IDcomposition},  $m(x\vert \mu)$  is infinitely divisible.

By Corollary~\ref{cor:limit} combined with~(\ref{eq:mix}), the density $\rho(x)=x\, \xi(x)$, 
where $\xi(x)$  is the L\'{e}vy  density of $m(x\vert \mu)$, is given by the limit 
\begin{align}
\mu\,\rho(x) = \lim_{n\to\infty}\, n\, x \, m(x\vert \tfrac{\mu}{n}) 
&= 
x \int_0^\infty f(x\vert y)\,  \lim_{n\to\infty}n\, g(y\vert \tfrac{\mu}{n}) \, dy  \nonumber \\
&= \mu\, x \int_0^\infty f(x\vert y)\, \ell(y) \, dy 
\label{eq:Levylimit} \\
\implies\quad \frac{\rho(x)}{x} \equiv \xi(x) &=   \int_0^\infty f(x\vert y)\, \ell(y) \, dy 
\end{align}
where $\ell(y)=r(y)/y$ is the L\'{e}vy density of $g(y\vert \mu)$.
\label{proof:marginal}
\end{proof}

Theorem~\ref{thm:marginal} holds for any pair of infinitely divisible densities $(f,g)$.
We now turn to a particular choice of $(f,g)$. 

\subsection{Stable/Gamma Case}
Let $f(x\vert y)$ be the stable density $f_\alpha(x\vert y)$ for $0<\alpha<1$ 
and  $g(y\vert \mu)$ the gamma density 
to give $m_{\alpha}(x\vert \mu,\lambda)$:
\begin{align}
 m_{\alpha}(x\vert \mu,\lambda)  &= \frac{\lambda^\mu}{\Gamma(\mu)} \int_0^\infty f_\alpha(x\vert y) \, y^{\mu-1}e^{-\lambda y}\, dy 
\end{align}
With $\eta(s)=s^\alpha$, $\psi(s)=\log(1+s/\lambda)$, the commutative diagram  of $m_{\alpha}(x\vert \mu,\lambda)$ is:
\begin{equation}
\begin{tikzpicture}[auto,scale=2.0, baseline=(current  bounding  box.center)]
\newcommand*{\size}{\scriptsize}%
\newcommand*{\gap}{.2ex}%
\newcommand*{\width}{1.8}%
\newcommand*{\height}{1.0}%

\node (P) at (0,0)  {$m_{\alpha}(x\vert \mu,\lambda)$};
\node (Q) at ($(P)+(\width,0)$) {$\left(\dfrac{\lambda}{\lambda+s^\alpha}\right)^\mu$};
\node (B) at ($(P)-(0,\height)$) {$\mu\, \alpha E_\alpha(-\lambda x^\alpha)$}; 
\node (C) at ($(B)+(\width,0)$) {$\mu\, \dfrac{\alpha s^{\alpha-1}}{\lambda+s^\alpha}$};   
\draw[Myarrow] ([yshift =  \gap]P.east)  --  
([yshift = \gap]Q.west) ;
\draw[Myarrow]([xshift  =  \gap]Q.south) --  
([xshift =  \gap]C.north);
\draw[Myarrow] ([xshift =  \gap]P.south) -- 
([xshift =  \gap]B.north);
\draw[Myarrow] ([yshift = -\gap]C.west) --  
([yshift = -\gap]B.east); 
\end{tikzpicture}
\label{eq:MLCD}
\end{equation}
We recognise $s^{\alpha-1}/(\lambda+s^\alpha)$ as the Laplace transform of 
the Mittag-Leffler function $E_\alpha(-\lambda x^\alpha)$.
Hence the entry $\mu\,\alpha E_\alpha(-\lambda x^\alpha)$ in the bottom left corner, 
which is arrived at by following the path of the commutative diagram involving Laplace  transforms.
 The equivalent, direct path from top left  to bottom left corner is given by~(\ref{eq:marginalLevy}) in Theorem~\ref{thm:marginal},
 where $\ell(y)=y^{-1} e^{-\lambda y}$ is the L\'{e}vy density of the gamma distribution and, by inspection,
$\xi(x)=\alpha E_\alpha(-\lambda x^\alpha)/x$.
Hence~(\ref{eq:marginalLevy}) becomes
\begin{align*}
\alpha\,E_\alpha(-\lambda x^\alpha) &= x \int_0^\infty f_\alpha(x\vert y) \, y^{-1} e^{-\lambda y} \,dy
\end{align*} 
which is precisely the assertion of Proposition~\ref{prop:main} that we sought to justify by appeal to infinite divisibility.

\section{Conclusion }
\label{sec:conclusion}

Pollard proved the complete monotonicity of the Mittag-Leffler function 
 using methods of complex analysis. 
 He also  cited personal communication  by Feller that he had discovered 
a proof based on ``methods of probability theory''.
In his published work, Feller  derived the result using the the  two-dimensional Laplace transform of a
 bivariate distribution involving the stable distribution on a positive variable.
As published, both Pollard's and Feller's approaches are actually  analytic rather than probabilistic,
despite the stable distribution appearing in Feller's approach and in the Pollard result itself. 

In this paper, we  adopted Bayesian reasoning  as the fundamental probabilistic approach to the problem.
We assigned a prior distribution to the scale factor of the stable distribution.
In particular, we discussed the assignment of a gamma distribution.
The special case  of the  exponential prior distribution reproduced the Feller result with ease.

Importantly, we discovered a novel integral representation of the Mittag-Leffler distribution.
With the aid of  the polynomially tilted stable density, we proceeded to prove the complete monotonicity  of a generalised 
Mittag-Leffler function by establishing  that it is the Laplace transform of the generalised  Mittag-Leffler distribution,
thereby generalising the Pollard result.

 The novel integral representation  arises from choosing the 
L\'{e}vy measure of the infinitely divisible gamma distribution as a prior distribution.
Accordingly, we presented a discussion of  infinite divisibility on the positive half-line, which led  to the discovery.
In this context, we have found it helpful to invoke a commutative diagram representation of  infinite divisibility.

On a philosophical note, we have taken ``methods of probability theory'' to refer to Bayesian reasoning,
placing an accent on distributions and the sum and product rules of probability theory.
This is  by no means to dismiss an  alternative approach based on 
products and powers of specified  random variables rather than direct assignment of distributions.
 We trust that  the Bayesian view will  nonetheless  find appeal amongst both probabilists,  
 to whom random variables are often the staple,
 and physicists, who routinely take an analytic  view in the study of Mittag-Leffler functions without invoking 
 an underlying random experiment or phenomenon.

\bibliography{MittagLeffler}{}
\bibliographystyle{plain} 
\end{document}